\documentclass[12pt]{amsart}
\usepackage{mathrsfs}
\usepackage{amsmath,amssymb,amsfonts,latexsym,txfonts}
\usepackage{eucal}
\usepackage{fancyhdr}

\newlength{\originalbase}
\newtheorem{theorem}{Theorem}

\input amssym.def
 
\input amssym.tex

\renewcommand{\SS}{\mathbb{S}}
\newcommand{\RR}{\mathbb{R}}

\begin{document}

\begin{center}
{\Large\bf Angles as probabilities}\\[4mm]
{\bf\em David V.~Feldman \\
Daniel A.~Klain}
\end{center}

\vspace{4mm}

Almost everyone knows that the inner angles of a triangle sum to $180^{\rm o}$.
But if you ask the typical mathematician how to sum the solid inner angles 
over the vertices of a tetrahedron, you are likely to receive 
a blank stare or a mystified shrug.  
In some cases you may be directed to the 
Gram-Euler relations for higher dimensional polytopes \cite{Grunbaum,Mull1,Shep-Perles,Sommer}, 
a 19th century result unjustly consigned to relative obscurity.  
But the answer is really much simpler than that, and here it is:
\begin{quotation}
The sum of the solid inner vertex angles of a tetrahedron $T$, \linebreak
divided by $2\pi$, gives the probability that the orthogonal
projection of $T$ onto a random 2-plane is a triangle.
\end{quotation}
How simple is that?  We will prove a more general theorem (Theorem~\ref{anglesumthm}) 
for simplices in $\RR^n$, but first consider the analogous
assertion in $\RR^2$.  The sum in radians of the angles of a triangle (2-simplex) $T$, 
when divided by the length $\pi$ of the unit semicircle,
gives the probability that the orthogonal
projection of $T$ onto a random line is a convex segment (1-simplex).  
Since this is {\em always} the case, the
probability is equal to 1, and the inner angle sum for
every triangle is the same.  
By contrast, a higher dimensional $n$-simplex may project 
one dimension down either to an $(n-1)$-simplex or to a lower dimensional
convex polytope having $n+1$ vertices.  
The inner angle sum gives a measure of how often each of these possibilities occurs.

Let us make the notion of ``inner angle" more precise.
Denote by $\SS^{n-1}$ the unit sphere in $\RR^n$ centered at the origin.  Recall that $\SS^{n-1}$ has
$(n-1)$-dimensional volume (i.e. surface area) $n\omega_n$,
where
$\omega_n = \tfrac{\pi^{n/2}}{\Gamma(1+n/2)}$
is the Euclidean volume of the unit ball in $\RR^n$.

Suppose that $P$ is a convex polytope in $\RR^n$, and let $v$ be any point of $P$.
The solid inner angle $a_P(v)$ of $P$ at $v$ is given by
$$a_P(v) = \{ u \in \SS^{n-1} \; | \; \ v+\epsilon u \in P \quad \hbox{ for some } \epsilon > 0 \}
$$
Let $\alpha_P(v)$ denote the measure of the solid angle $a_P(v) \subseteq \SS^{n-1}$, 
given by the usual surface area measure on subsets of the sphere.  For the moment we are 
primarily concerned with values of $\alpha_P(v)$ when $v$ is a {\em vertex} of $P$.

If $u$ is a unit vector, then denote by $P_u$ the orthogonal projection of $P$ onto the subspace $u^\perp$ in $\RR^n$.  
Let $v$ be a vertex of $P$.  The projection $v_u$ lies in the relative interior of $P_u$ if and only if
there exists $\epsilon \in (-1,1)$ such that $v+\epsilon u$ lies in the interior of $P$.  This holds
iff $u \in \pm a_P(v)$.  If $u$ is a random unit vector in $\SS^{n-1}$, then
\begin{equation}
\mathtt{Probability[}v_u \in \mathrm{relative \; interior}(P_u)\mathtt{]} = \frac{2\alpha_P(v)}{n\omega_n},
\label{pv}
\end{equation}
This gives the probability that $v_u$ is no longer a vertex of $P_u$.

For {\em simplices} we now obtain the following theorem.
\begin{theorem}[Simplicial Angle Sums]
Let $\Delta$ be an $n$-simplex in $\RR^n$, and let $u$ be a random unit vector.  
Denote by $p_\Delta$ the probability that the orthogonal projection $\Delta_u$ is an $(n-1)$-simplex.  Then
\begin{equation}
p_\Delta = \frac{2}{n\omega_n} \sum_v \alpha_\Delta(v)
\label{anglesum}
\end{equation}
where the sum is taken over all vertices of the simplex $\Delta$.
\label{anglesumthm}
\end{theorem}

\begin{proof}  Since $\Delta$ is an $n$-simplex, $\Delta$ has $n+1$ vertices, and a projection $\Delta_u$
has either $n$ or $n+1$ vertices.  (Since $\Delta_u$ spans an affine space of dimension $n-1$, it cannot have fewer than $n$ vertices.)
In other words, either exactly 1 vertex of $\Delta$ falls to the relative interior of $\Delta_u$, so that
$\Delta_u$ is an $(n-1)$-simplex, or none of them do.  
By the law of alternatives, the probability $p_\Delta$ is now given by the sum of the probabilities~(\ref{pv}),
taken over all vertices of the simplex $\Delta$.
\end{proof}

The probability~(\ref{anglesum}) is always equal to 1 for 2-dimensional simplices (triangles).
For 3-simplices (tetrahedra) the probability may take any value $0 < p_\Delta < 1$.  To obtain a value
closer to one, consider the convex hull of an equilateral triangle in $\RR^3$ with a point outside the triangle,
but very close to its center.  To obtain $p_\Delta$ close to zero, consider the convex hull of two skew 
line segments in $\RR^3$ whose centers are very close together (forming a tetrahedron that is almost
a parallelogram).  Similarly, for
$n \geq 3$ the solid vertex angle sum of an $n$-simplex varies within a range
$$0 < \sum_v \alpha_T(v) < \frac{n\omega_n}{2}.$$
Equality at either end is obtained only if one allows for the degenerate limiting cases.
These bounds were obtained earlier by Gaddum \cite{Gaddum1,Gaddum2}, using a more complicated (and 
non-probabilistic) approach.

Similar considerations apply 
to the solid angles at arbitrary faces of convex polytopes.
If $F$ is a face of a convex polytope $P$, denote 
the centroid of $F$ by $\hat{F}$, and
define the solid inner angle measure $\alpha_P(F) = \alpha_P(\hat{F})$, using the definition
above for the solid angle at a {\em point} in $P$.
In analogy to~(\ref{pv}), we have
\begin{equation}
\mathtt{Probability[} \hat{F}_u \in \mathrm{relative \; interior}(P_u)\mathtt{]} 
= \frac{2\alpha_P(F)}{n\omega_n},
\label{pvk}
\end{equation}
Omitting cases of measures zero, this gives the 
probability that a proper face $F$ is no longer a face of $P_u$.
(Note that $\dim F_u  = \dim F$ for all directions $u$ except a set of measure zero.) 
Taking complements, we have
\begin{equation}
\mathtt{Probability[} F_u \hbox{ is a proper face of } P_u \mathtt{]} \; = \; 1 - \frac{2\alpha_P(F)}{n\omega_n}.
\label{pfc}
\end{equation}
For $0 \leq k \leq n-1$, denote by $f_k(P)$ the number of $k$-dimensional faces of a polytope $P$.
The sum of the probabilities~(\ref{pfc}) gives the expected number of $k$-faces of the projection
of $P$ onto a random hyperplane $u^\perp$; that is,
\begin{equation}
\mathtt{Exp[} f_k(P_u) \mathtt{]} \; = \; \sum_{\dim F = k} \left( 1 - \frac{2\alpha_P(F)}{n\omega_n} \right) 
\; = \; f_k(P) - \frac{2}{n\omega_n}\sum_{\dim F = k} \alpha_P(F),
\label{expk}
\end{equation}
where the middle sum is taken over $k$-faces $F$ of the polytope $P$.  

If $P$ is a convex polygon in $\RR^2$, then $P_u$ is always a line segment with exactly 2 vertices,
that is, $f_0(P_u) = 2$.  In this case the expectation identity~(\ref{expk}) yields the familiar
$$\sum_v \alpha_P(v) \; = \; \pi(f_0(P) - 2).
$$
If $P$ is a convex polytope in $\RR^3$, then $P_u$ is a convex polygon, 
which always has exactly as many vertices as edges; that is, $f_0(P_u) = f_1(P_u)$.  Therefore,
$\mathtt{Exp[} f_0(P_u) \mathtt{]} = \mathtt{Exp[} f_1(P_u) \mathtt{]}$, and 
the expectation identities~(\ref{expk}) imply that
$$\frac{1}{2\pi}\sum_{vertices \; v} \alpha_P(v) - \frac{1}{2\pi}\sum_{edges \; e} \alpha_P(e) 
\; = \; f_0(P) - f_1(P) \; = \; 2 - f_2(P).$$
where the third equality follows from the classical Euler formula $f_0 - f_1 + f_2 = 2$
for convex polyhedra in $\RR^3$.

These arguments were generalized by Perles and Shephard \cite{Shep-Perles}
(see also \cite[p. 315a]{Grunbaum})
to give a simple proof of the classical Gram-Euler identity for convex polytopes:
\begin{equation}
\sum_{F \subseteq \partial P} (-1)^{\dim F} \alpha_P(F) = (-1)^{n-1} n \omega_n,
\label{g-e-id}
\end{equation}
where the sum is taken over all proper faces $F$ of an
$n$-dimensional convex polytope $P$.  
In the general case one applies
the additivity of expectation to alternating sums over $k$ of the identities~(\ref{expk}), obtaining identities that relate the Euler numbers of the boundaries $\partial P$ and $\partial P_u$. 
Since the $\partial P$ is an piecewise-linear $(n-1)$-sphere, while $\partial P_u$ is a 
piecewise-linear $(n-2)$-sphere,
these Euler numbers are easily computed, and~(\ref{g-e-id}) follows.

The Gram-Euler identity~(\ref{g-e-id}) can be viewed as a discrete analogue of the Gauss-Bonnet theorem,
and has been since generalized to Euler-type identities for angle sums over polytopes 
in spherical and hyperbolic spaces \cite{Grunbaum,Mull1,Sommer}, as well as for mixed volumes 
and other valuations on polytopes \cite{Mull3}.

\vspace{2mm}
\noindent David V.~Feldman: {\em Dept.~of Mathematics, University of New Hampshire, Durham, NH 03824 USA, {\tt David.Feldman@unh.edu}} 

\vspace{2mm}
\noindent Daniel A.~Klain: {\em Dept.~of Mathematical Sciences, 
University of Massachusetts Lowell, Lowell, MA 01854 USA,
{\tt Daniel\_Klain@uml.edu}}

\end{document}